\documentclass[11pt]{article}

\usepackage{amsmath,amssymb,amsthm}
\usepackage[dvips]{graphicx}
\usepackage{latexsym}
\usepackage{eucal}
\usepackage[pdftex]{hyperref}
\usepackage[latin1]{inputenc}
\usepackage[english,activeacute]{babel}

\newtheorem{theorem}{Theorem}[section]

\newtheorem{lemma}[theorem]{Lemma}
\newtheorem{corollary}[theorem]{Corollary}	

\theoremstyle{definition}

\newtheorem{remark}[theorem]{Remark}
\newtheorem{definition}[theorem]{Definition}

\title{{\bf\Large Problems with  mixed boundary conditions in Banach spaces}}
\author{{\large Dionicio Pastor Dallos Santos}\footnote{Email: dionicio@ime.usp.br}\hspace{2mm}
{\bf\large}\vspace{1mm}\\
{\small Department of Mathematics, IME-USP, Cidade Universit\'aria,}\\
{\small  CEP 05508-090, S\~ao Paulo, SP, Brazil}
}

\date{}
\begin{document}
\maketitle

\begin{abstract}
Using Leray-Schauder degree or  degree  for $\alpha$-condensing maps  we obtain  the existence of at least one solution for the  boundary  value problem of the type
\[
\left\{\begin{array}{lll}
(\varphi(u' ))'   = f(t,u,u') & & \\
u(T)=0=u'(0),  & & \quad \quad 
\end{array}\right.
\] 
where $\varphi: X\rightarrow X   $ is a homeomorphism  with  reverse Lipschitz  such that $\varphi(0)=0$, $f:\left[0, T\right]\times X \times X \rightarrow X $ is a continuous function,  $T$ a positive real number and $X$ is  a real  Banach space.
\end{abstract}
 
\medskip




\section{Introduction}
The purpose of this article is to obtain  some existence results for  the nonlinear boundary value problem of the form
\begin{equation}\label{equa1}
\left\{\begin{array}{lll}
(\varphi(u' ))'   = f(t,u,u') & & \\
u(T)=0=u'(0),
\end{array}\right.
\end{equation}
where $\varphi: X \rightarrow  X   $ is a homeomorphism  such  that $\varphi(0)=0$  and  $\varphi^{-1}$  is Lipschitz, $f:\left[0, T\right]\times X \times X\rightarrow X $ is a continuous function,  $T$ a positive real number, and  $X$ is a real Banach space. We call \textsl{solution } of this problem any function $u:\left[0, T\right]\rightarrow X$ of class $C^{1}$ such that  the function  $t\mapsto \varphi(u'(t))$ is continuously differentiable, satisfying the  boundary conditions and $(\varphi(u'(t) ))' = f(t,u(t),u'(t))$ for all $t\in\left[0,T\right]$.

The existence of solutions for second-order boundary value problems has been studied by many authors using various methods (see \cite{ chandra, guo1, monch, szufula, zhou})

In particular, the authors in  \cite{chandra} have studied the following boundary value problem: 
\begin{equation}\label{antigo}
\left\{\begin{array}{lll}
u''   = f(t,u, u') & & \\
au(0)-bu'(0)=u_{0}, \ \ cu(1)+du'(1)=u_{1},
\end{array}\right.
\end{equation}
where $a,b,c,d$  and $ad+bc>0$. They obtained the existence of solutions of (\ref{antigo}) using Darbo fixed point theorem and properties of the measure of noncompactness.

Recently, W.-X. Zhou and J. Peng \cite{zhou} have studied the  following boundary value problem:
\begin{equation}\label{bouches}
\left\{\begin{array}{lll}
-u'' = f(t,u) & & \\
u(0)=0=u(1)
\end{array}\right.
\end{equation}
where $f:\left[0, 1\right]\times X \rightarrow X $ is a continuous function and  $X$ is a Banach space. They obtained the existence of solutions of (\ref{bouches}), where the main tools used in the study are Sadovskii fixed point theorem  and precise computation of measure of noncompactess.

Inspired by these results, the main aim of this paper is  to study  the existence of at least one solution for the  boundary  value problem (\ref{equa1}) using  Leray-Schauder degree or  degree  for $\alpha$-condensing maps. For this, we reduce the nonlinear boundary value problem to some fixed points problem. Next, we shall essentially consider two types of regularity assumptions for $f(t,x,y)$. In Theorem \ref{numero1} we suppose that $f$ is completely continuous, which allows us to prove that the associated  fixed point operator is completely continuous  required by a Leray-Schauder approach. In Theorem \ref{numero2} we only assume some regularity conditions expresed in terms of the measure of noncompactness, which allows us to apply the methods of topological degree theory for $\alpha$-condensing maps.

The  paper is organized as follows. In Section 2, we establish the notation, terminology, and various lemmas which will be  used throughout this paper. Section 3,  we formulate the fixed point operator equivalent to the problem (\ref{equa1}). Section 4, we give main results in this paper. Section 5, we study the existence of at least one solution for (\ref{equa1}) in Hilbert spaces. For these results, we adapt the ideas of  \cite{ma},  \cite{man11}  and  \cite{man} to the present situation.


\section{Notations and preliminary results}
\label{S:-1}
We first introduce some notation. For fixed $T$, we denote  the usual norm in $L^{1}=L^{1}(\left[0,T\right],  X)$ for $\left\| \cdot \right\|_{L^{1}}$. For $C=C(\left[0,T\right],X)$ we indicate the Banach space of all  continuous functions from $\left[0, T\right]$ into $X$ witch the norm  $\left\| \cdot \right\|_{\infty}$ and for  $C^{1}=C^{1}(\left[0,T\right], X)$  we designate the Banach space of continuously differentiable functions  from $\left[0, T\right]$ into $X$ endowed with the usual norm $\left\|u\right\|_{1}= $max$\left\{ \left\|u\right\|_{\infty}, \  \left\|u' \right\|_{\infty}\right\}$.

We introduce the following applications:

the  \textit{Nemytskii operator} $N_f:C^{1} \rightarrow C$, 
\begin{center}
 $N_f (u)(t)=f(t,u(t),u'(t))$, 
\end{center}
the  \textit{integration operator}   $H:C \rightarrow C^{1}$, 
\begin{center}
$ H(u)(t)=\int_0^t u(s)ds$, 
\end{center}
the  continuous linear application   $K:C \rightarrow C^{1}$,
\begin{center}
$ K(u)(t)= -\int_t^T u(s)ds$. 
\end{center} 
 
Throughout this paper, we denote  $(X,\left\|\cdot  \right\|) $  a  real Banach space and $I=\left[0,T\right]$. For $A\subseteq C^{1}$, we  use the notation;
\begin{center}
$A(t)=\left\{u(t):u\in A\right\}$,

$A(I)=\left\{u(t):u\in A, \  t\in I \right\}$,

$A'(t)=\left\{u'(t):u\in A\right\}$,

$A'=\left\{u':u\in A\right\}$,

$A'(I)=\left\{u'(t):u\in A, \  t\in I \right\}$.
\end{center}

\begin{definition}
 Let $X$ be a Banach space and  let $M_{X}$ be the family of bounded subsets of $X$. The Kuratowski measure of noncompactness is the map $\alpha: M_{X} \rightarrow [0,\infty)$ defined for
\begin{center}
$\alpha(B)=$inf$[d>0: B \ $admits a finite cover by sets of diameter $\leq d]$;  here $B\in M_{X}$.
\end{center}
\end{definition}
\textbf{Properties:}
\begin{itemize}
\item [(a)] $\alpha(B)=0$ iff $\overline{B}$ is compact.
\item [(b)] $S\subset B$  then  $\alpha(S)\leq \alpha(B)$.
\item [(c)] $\alpha(\overline{B})=\alpha(B)$.
\item [(d)] $\alpha(B\cup S)= $max $\left\{\alpha(B), \alpha(S)\right\}$.
\item [(e)] $\alpha(\lambda B)=\left|\lambda \right| \alpha( B)$, where  $\lambda \in \mathbb R$ and $\lambda B=\left\{\lambda b: b\in B\right\}$.
\item [(f)] $\alpha(B+ S)\leq \alpha(B)+\alpha(S)$, where $B+S=\left\{b+s: b\in B, s\in S\right\}$.
\item [(g)] $\alpha(\overline{conv}(B))=\alpha(B)$.
\end{itemize}

The details of $\alpha$ and its properties can be found in \cite{man7}.

\begin{definition}(see \cite{guo}). Assume that $D\subset X$ the  mapping $A:D \rightarrow X$ is said to be a condensing operator if $A$ is continuous, 
bounded (sends bounded sets into bounded sets), and for any nonrelatively compact and bounded set $S\subset D$,
\begin{center}
$\alpha(A(S))< \alpha(S)$.
\end{center}
\end{definition}
 
The following lemmas are of great importance in the proof of our main results. The proofs can be found in  \cite{guo}.

In the following, we denote $\alpha_{c}$ and $\alpha_{1}$ by the noncompactness measure in $C$ and $C^{1}$, respectively.
\begin{lemma}\label{lem1}
Let $S$ be a bounded subset of  real numbers and  $B$ a  bounded subset of  $X$. Then     
\begin{center}
$\alpha(SB)=\left( \displaystyle \sup_{t\in S}\left|t\right|\right) \alpha(B)$,
\end{center}
where $SB=\left\{sb:s\in S, \  b\in B\right\}$.
\end{lemma}
\begin{lemma}\label{lem2}
Let $A, \  B $  be  bounded subsets of Banach spaces  $X$ and $Y$ respectively with
\begin{center}
$\left\|(x, y)\right\| = \max \left\{ \left\|x\right\|, \  \left\|y\right\| \right\}$.
\end{center}
 Then     
\begin{center}
$\alpha(A\times B)=\max \left\{\alpha(A), \ \alpha(B)\right\}$.
\end{center}
\end{lemma}
\begin{lemma}\label{lem3} If $H\subset C$ is bounded and equicontinuous, then we have the following:
\begin{itemize}
\item [(1)] $\alpha_{c}(H)=\alpha(H(I))$.
\item [(2)] $\alpha(H(I))=\displaystyle \max_{I}\alpha(H(t))$.
\end{itemize}
\end{lemma}
\begin{lemma}\label{lem4}
If $H$ is a bounded set in  $C^{1}$, then    
\begin{itemize} 
\item [(ii)] $\alpha_{1}(H)\geq \alpha(H(I))$.
\item [(ii)] $2\alpha_{1}(H)\geq \alpha(H'(I))$.
\end{itemize}
\end{lemma}
\begin{lemma}\label{lem5}
If $H$ is a bounded set in  $C^{1}$ and $H'$ equicontinuous, then    
\begin{center}
$\alpha_{1}(H)=  \max \left\{  \displaystyle \max_{I}\alpha(H(t)), \   \displaystyle \max_{I}\alpha(H'(t))  \right\} $.
\end{center}
\end{lemma}


\section{Fixed point formulations}
\label{S:0}
Let us consider the operator 
\begin{center}
$M_{1}:C^{1} \rightarrow C^{1}$,

$u \mapsto K\left(\varphi^{-1}\left[H(N_f(u))\right]\right)$.
\end{center}
 Here $\varphi^{-1}$ is understood as the operator  $\varphi^{-1}: C \rightarrow C$ defined for $\varphi^{-1}(v)(t)=\varphi^{-1}(v(t))$. It is clear that $\varphi^{-1}$ is continuous and sends bounded sets into bounded sets.

\begin{lemma}\label{mate2}
$u \in C^{1}$ is  a solution  of (\ref{equa1}) if and only if  $u$ is a fixed point of the operator $M_{1}$.
\end {lemma} 
\begin{proof}
Let $u$ be a solution of (\ref{equa1}). This implies that 
\begin{center}
$(\varphi(u' ))' = f(t,u,u'), \  \  u(T)=0=u'(0)$.
\end{center}
Integrating  of $0$ to $t$ and using the fact that $u'(0)=0$, we deduce that
\begin{center}
$\varphi(u'(t))= H(N_{f}(u))(t)$.
\end{center}
Applying $\varphi^{-1}$ and $K$ to both of its members and using that $u(T)=0$, we have that
\begin{align*}
u(t)&= K\left( \varphi^{-1} \left[H(N_{f}(u))\right]\right)(t) \\
&=M_{1}(u)(t) \  \   \ (t\in \left[0,T\right].
\end{align*}
Conversely, since, by definition of the mapping $M_{1}$,
\begin{center}
$K\left( \varphi^{-1} \left[H(N_{f}(u))\right]\right)(T)=0,   \  \  \varphi^{-1} \left[H(N_{f}(u))\right](0)=0$,
\end{center}
it is a simple matter to see that if $u$ is such that $u=M_{1}(u)$ then $u$ is a solution to (\ref{equa1}).
\end{proof}

Using the theorem of Arzel\`a-Ascoli we show that the operator  $M_{1}$ is completely continuous.
\begin{lemma}\label{lema1}
 If $f$  is completely continuous, then  the operator  $M_{1}:C^{1} \longrightarrow C^{1}$ is completely continuous.
\end{lemma}
\begin{proof}
Let $ \Lambda \subset C^{1}$ be a bounded set. Then, if $u \in \Lambda$, there exists a constant  $\rho>0$ such  that  
\begin{equation}\label{dp4}
\left\| u\right\|_{1}\leq \rho.
\end{equation}
Next, we show that  $\overline{M_{1}(\Lambda)}\subset  C^{1}$ is a compact set. Let $(v_{n})_{n} $   be a sequence in   $M_{1}(\Lambda)$, and 
 let $(u_{n})_{n}$ be  a sequence  in  $\Lambda$ such that  $v_{n}=M_{1}(u_{n})$. Using (\ref{dp4}), we have that there exists a constant  $W>0$ such that, for all $n\in \mathbb{N}$,   
\begin{center}
 $\left\| N_{f}(u_{n})\right\|_{\infty}\leq W$,
\end{center}
which implies that
\begin{center}
$\left\| H(N_{f}(u_{n})) \right\|_{\infty}\leq WT$.
\end{center}
Hence the sequence  $ \left(H(N_{f}(u_{n}))\right)_{n}$ is bounded in  $C$. Moreover, for  $t, t_{1}\in\left[0, T\right]$ and for all $n\in \mathbb{N}$, we have that 
\begin{align*}
& \left\|H(N_{f}(u_{n}))(t) - H(N_{f}(u_{n}))(t_{1})\right\| \\
& =\left\|\int_{0}^{t} f(s,u_{n}(s),u_{n}'(s))ds - \int_{0}^{t_{1}} f(s,u_{n}(s),u_{n}'(s))ds\right\|\\
&= \left\|\int_{t_1}^t f(s,u_n(s),u'_n(s))ds\right\|\\
&\leq W\left|t-t_1\right|, 
\end{align*} 
which  implies that  $ \left(H(N_{f}(u_{n}))\right)_{n}$ is equicontinuous.

On the other hand, for $t\in \left[0,T\right]$
\begin{center}
$B(t)=\left\{H(N_{f}(u_{n}))(t):n\in \mathbb{N} \right\}$,
\end{center}
where
\begin{align*}
H(N_{f}(u_{n}))(t)&=\int_{0}^{t} N_{f}(u_{n})(s)ds \\
&=\int_{0}^{t} f(s,u_{n}(s),u_{n}'(s))ds \\
&=t\displaystyle \lim_{m \to \infty}\displaystyle\sum_{k=1}^{m} f(s_{k},u_{n}(s_{k}),u'_{n} (s_{k}))\frac{(s_{k}-s_{k-1})}{t}.
\end{align*}
Recalling that the convex hull of a set $ A \ \subseteq  X$   is given by
\begin{center}
$conv(A)= \left\{ \displaystyle\sum_{i=1}^N \alpha_{i}x_{i}: x_{i}\in A, \  \alpha_{i}\in \mathbb R, \  \alpha_{i}\geq 0, \  \displaystyle\sum_{i=1}^N \alpha_{i}=1  \right\}$, 
\end{center}
it follows that
\begin{center}
$\int_0^t f(s,u_{n}(s),u'_{n}(s))ds \in t\overline{conv}(\left\{f(s, u_{n}(s), u'_{n}(s)):s\in[0,T],\  n\in \mathbb{N}\right\})$,
\end{center}
which implies that
\begin{center}
$B(t)=\left\{H(N_{f}(u_{n}))(t):n\in \mathbb{N} \right\}\subset t\overline{conv}(\left\{f(s,u_{n}(s),u'_{n}(s)):s\in[0,T], \ n\in \mathbb{N}\right\})$.
\end{center}
Using the fact that  $f:\left[0, T\right]\times X \times X \longrightarrow X $ is completely continuous,  we deduce that $\alpha(B(t))=0$. Hence, $B(t)$ is a relatively compact set in $ X$. Thus, by the Arzel\`a-Ascoli theorem  there is a subsequence of $(H(N_{f}(u_{n})))_{n}$, which we call $( H(N_{f}(u_{n_{j}})))_{j}$, which is convergent in C. Using the fact that $\varphi^{-1}: C\rightarrow  C$ is continuous it follows from
\begin{center}
$M_{1}(u_{n_{j}})'=\varphi^{-1} \left[H(N_{f}(u_{n_{j}}))\right] $
\end{center}
that the sequence $(M_{1}(u_{n_{j}})')_{j}$  is convergent in $C$ and hence $(v_{n_{j}})_{j}=( M_{1}(u_{n_{j}}))_{j}$ is convergent in $C^{1}$. Finally, let  $(v_{n})_{n}$ be a sequence in  $\overline{M_{1}(\Lambda)}$. Let  $(z_{n})_{n}\subseteq M_{1}(\Lambda)$  be such that
\[
\lim_{n \to \infty}\left\| z_{n}-v_{n}\right\|_{1}=0.
\]
Let    $(z_{n_{j}})_{j}$  be a subsequence of  $(z_{n})_{n}$  such that  converge to $z$. It follows that  $z\in \overline{M_{1}(\Lambda)}$ and  $(v_{n_{j}})_{j}$ converge to $z$. This concludes the proof.
\end{proof}
In order to apply Leray-Schauder degree to the operator $M_{1}$, we introduced a family of problems
depending on a parameter $\lambda $. For, $\lambda \in \left[0,T\right]$, we consider the family of boundary value problems
\begin{equation}\label{infi2}
\left\{\begin{array}{lll}
(\varphi(u' ))'   = \lambda f(t, u(t), u'(t) & & \\
u(T)=0=u'(0). 
\end{array}\right.
\end{equation}
Notice that (\ref{infi2}) coincide with (\ref{equa1}) for $\lambda =1$. So, for  each $ \lambda \in [0,1]$, the operator associated to  \ref{infi2}  for Lemma \ref{mate2}  is the  operator $M(\lambda,\cdot)$, where $M$  is defined on $[0,1]\times C^{1}$  by 
\begin{center}
$ M(\lambda,u)= K\left(\varphi^{-1} \left[ \lambda H( N_f (u))\right] \right)$.
\end{center}
Using the same arguments as in the proof  of Lemma \ref{lema1}  we show that the operator  $M$ is completely continuous. Moreover, using the same reasoning as above,  the system (\ref{infi2}) (see Lemma \ref{mate2}) is equivalent to the problem 
\begin{equation}\label{igualdade}
u=M(\lambda, u).
\end{equation}


\section{ Main results}
\label{S:2}
In this section, we present and prove our main results.
\begin{theorem}\label{numero1}
 Let $X$ be a Banach space, and $\varphi^{-1}$  a homeomorphism  with Lipschitz constant $k$. Suppose that  $f$  is completely continuous and that there exist two numbers $c_{0}, c_{1}\geq 0$ such that   
\begin{center}
$\left\|f(t,x,y)\right\|\leq c_{0}  + c_{1}\left\|y \right\|$,  \ \  for all  \ \  $(t,x,y) \in [0,T]\times X \times X$.
\end{center} 
 Then problem  (\ref{equa1}) has at least one solution.
\end{theorem}
\begin{proof}
Let $(\lambda,u)\in [0, 1]\times C^{1}$ be such that  $M(\lambda,u)=u$. Using   \ref{igualdade}  we have that $u$ is solution  of (\ref{infi2}), which implies that
\begin{center} 
 $u'= \varphi^{-1}\left[\lambda H(N_f(u))\right]$, \   $u'(0)=0=u(T)$.
 \end{center}
Using the fact that  $\varphi^{-1}$ is a homeomorphism with Lipschitz constant $k$, we deduce that
\begin{center}
$\left\|u'(t)\right\| \leq kc_{0}T + kc_{1} \int_0^t \left\|u'(s)\right\|ds   \quad  (t\in [0,T])$.                                                     
\end{center}
 By  Gronwall's Inequality, we have
\begin{center}
$\left\|u'(t)\right\|\leq kc_{0}Te^{\int_0^t kc_{1} ds}\leq kc_{0}Te^{kc_{1}T}   \quad  (t\in [0,T])$.
\end{center}
Hence,  $\left\|u' \right\|_{\infty}\leq kc_{0}Te^{kc_{1}T}:= \beta$. Because $u\in C^{1}$ is such that $u'(T)=0$ we have that 
\begin{center}
$\left\|u(t)\right\|\leq \int_t^T\left\|u'(s)\right\|ds \leq \int_0^T \left\|u'(s)\right\|ds\leq \beta T   \  \  \  (t\in [0,T])$,
\end{center}
 and hence 
\begin{center}
$\left\| u \right\|_{1}\leq R_{1}$, \  where  $R_{1}=$max$ \left\{\beta, \ \beta T\right\} $.     
\end{center}
Using  that $M$ is completely continuous  we deduce that for each $\lambda\in [0, 1]$, the Leray-Schauder  degree $deg_{LS}(I-M(\lambda,\cdot),B_{\rho}(0),0)$ is well-defined  for any $\rho>R_{1} $, and by the homotopy invariance  we have that
\begin{center}
$deg_{LS}(I-M(1 ,\cdot),B_{\rho}(0) ,0)= deg_{LS}(I-M(0,\cdot),B_{\rho}(0),0)$.
\end{center}
Hence,  $deg_{LS}(I-M(1,\cdot),B_{\rho}(0),0)\neq 0$. This, in turn, implies that there exists  $u\in B_{\rho}(0) $  such that  $M_{1}(u)=u$, which is a solution for (\ref{equa1}).
\end{proof}

\begin{theorem}\label{numero2}
 Let $X$ be a Banach space, and $\varphi^{-1}$  a homeomorphism  with Lipschitz constant $k$. Assume that  $f$ is continuous and  satisfies the following conditions.   
\begin{enumerate}
\item There exist two numbers $c_{0}, c_{1}\geq 0$ such that   
\begin{center}
$\left\|f(t,x,y)\right\|\leq c_{0}  + c_{1}\left\|y \right\|$,  \ \  for all  \ \  $(t,x,y) \in [0,1]\times X \times X$.
\end{center} 
\item  For all bounded subsets  $A, B$ in  $X$,	
\begin{center}
$ \alpha(f( \left[0, 1 \right] \times A \times B))\leq k_{1} \max \left\{ \alpha(A), \ \alpha(B)\right\}$, \  where  \  $0<k_{1} < 1/2k$.
\end{center}
\end{enumerate} 
Then problem  (\ref{equa1}) has at least one solution.
\end{theorem}
\begin{proof}
Observe that $M_{1}$ maps bounded sets into bounded sets. Furthermore, its continuity follows by the continuity of the operators which compose $M_{1}$. We show that the operator $M_{1}$ is condensing ($\alpha$-condensing). In fact, for a bounded set $\Lambda$ in $C^{1}$, there exists  a constant $L_{1}>0$ such that  
\begin{center}
$\left\| N_{f}(u)\right\|_{\infty}\leq L_{1}$, \ for \ all \  $u \in \Lambda$. 
\end{center}
For $t, t_{1}\in\left[0, 1\right]$ we have that
\begin{align*}
\left\|(M_{1}u)'(t)-(M_{1}u)'(t_{1})\right\| &=\left\| \varphi^{-1}\left[H(N_{f}(u))\right](t) - \varphi^{-1} \left[H(N_{f}(u_{n}))\right](t_{1}) \right\|\\
&\leq k\left\| H(N_{f}(u))(t) -  H(N_{f}(u))(t_{1})\right\|\\
&\leq k\left\|\int_{t_1}^t f(s,u(s),u'(s))ds\right\|\\
&\leq kL_{1}\left|t-t_1\right|, 
\end{align*} 
which means $(M_{1}\Lambda)'$  is equicontinuous. Applying Lemma \ref{lem5} there exists $\tau \in \left[0, 1\right]$ or  $\omega \in \left[0, 1\right] $ with 
\begin{center}
$\alpha_{1}(M_{1}\Lambda)= \alpha((M_{1}\Lambda)(\tau))$
\end{center}
or
\begin{center}
$\alpha_{1}(M_{1}\Lambda) = \alpha((M_{1}\Lambda)'(\omega)) $.
\end{center}
Let us  consider the first case.
\begin{align*}
\alpha_{1}(M_{1}\Lambda)&= \alpha((M_{1}\Lambda)(\tau))\\
&= \alpha \left(\left\{ K\left(\varphi^{-1}\left[H(N_{f}(u))\right]\right)(\tau):u\in \Lambda  \right\}\right). 
\end{align*} 
Using the properties of $\alpha$, we see that
\begin{align*}
\alpha_{1}(M_{1}\Lambda)&\leq \alpha  \left(\left\{\int_\tau^1 \varphi^{-1} \left[H(N_{f}(u))(s)\right] ds:  u \in  \Lambda \right\}\right) \\
&\leq (1-\tau)\alpha \left( \overline{conv}\left\{ \varphi^{-1} \left[ H(N_{f}(u))(s)\right] : s\in \left[\tau, 1\right], \  u\in \Lambda \right\}\right) \\ 
&\leq \alpha \left( \left\{\varphi^{-1} \left[H(N_{f}(u))(s)\right]:  s\in \left[0, 1\right], \  u\in \Lambda \right\}  \right).
\end{align*}
Using the fact that $\varphi^{-1}$ is a homeomorphism with Lipschitz constant $k$, we deduce that
\begin{align*}
\alpha_{1}(M_{1}\Lambda)&\leq k \alpha \left( \left\{ H(N_{f}(u))(s): s\in \left[0, 1\right], \  u\in \Lambda  \right\}  \right) \\
&\leq k \alpha \left(  \left\{ \int_0^s f(t, u(t),u'(t)) : s\in \left[0, 1\right], \  u\in \Lambda \right\} \right) \\ 
&\leq k\alpha \left(\left[0, 1\right]\overline{conv} \left\{ f(t, u(t),u'(t)) : t\in \left[0, 1\right], \  u\in \Lambda \right\} \right).
\end{align*} 
Applying  Lemma  \ref{lem1} and  again the properties of $\alpha$, we obtain that
\begin{align*}
\alpha_{1}(M_{1}\Lambda)&\leq k \alpha \left(\left\{f(t, u(t),u'(t)) : t\in \left[0, 1\right], \  u\in \Lambda \right\}  \right) \\
&\leq k  \alpha \left( f\left( \left[0,1\right] \times \Lambda \left( \left[0,1\right] \right)\times \Lambda' \left( \left[0,1\right] \right) \right) \right). 
\end{align*}  
Using the assumption 2, we have that
\begin{center}
$\alpha_{1}\left( M_{1}\Lambda \right)\leq k k_{1} \max \left\{ \alpha \left(\Lambda(\left[0, 1\right])\right), \ \alpha \left(\Lambda'(\left[0, 1\right])\right)\right\}$.
\end{center}
This implies, by Lemma \ref{lem4}
\begin{center}
$\alpha_{1}\left(M_{1}\Lambda\right)\leq 2kk_{1}\alpha_{1}(\Lambda)$.
\end{center}

Consider the alternative case. Proceeding as before, we obtain
\begin{center}
$\alpha_{1}(M_{1}\Lambda) =\alpha\left((M_{1}\Lambda)'(\omega)\right) \leq 2kk_{1}\alpha_{1}(\Lambda)$.
\end{center}
Therefore, in either case, we obtain 
\begin{center}
$\alpha_{1}(M_{1}\Lambda)\leq 2kk_{1}\alpha_{1}(\Lambda)$.
\end{center}
 By  the assumption 2, we get $0<2kk_{1}<1$, therefore  $M_{1}$  is $\alpha$-condensing.

Let us consider the function
\begin{center}
$\widetilde{M}:\left[0, 1\right]\times C^{1}\rightarrow C^{1}, \quad (\lambda,u)\mapsto \lambda M_{1}(u)$.
\end{center}
Let $(\lambda,u)\in \left[0, 1\right]\times C^{1} $ be such that $u=\widetilde{M}(\lambda,u)$. Using the fact that $\varphi^{-1}$ is a homeomorphism with Lipschitz constant $k$ and   Gronwall's Inequality, we deduce that there exists a constant $r>0$ such that $\left\|u\right\|_{1}< r$.

Finally, we show the existence of at least one solution of (\ref{equa1}) using the homotopy invariance of the degree for $\alpha$-condensing maps. Let $B$ be bounded in $C^{1}$. Then
\begin{align*}
\alpha_{1}\left( \widetilde{M}\left(\left[0,1\right]\times B\right)\right)&=\alpha_{1}\left( \widetilde{M}(\lambda, u ):\lambda \in \left[0, 1\right],\ u\in B\right) \\
&\leq 2kk_{1}\alpha_{1}\left(B\right). 
\end{align*}  
Then we have that for each $\lambda \in \left[0,1\right]$, the degree  $deg_{N}(I-\widetilde{M}(\lambda, \cdot), B_{r}(0), 0)$  is well-defined and, by the properties of that degree, that 
\begin{center}
 $deg_{N}(I-\widetilde{M}(1, \cdot), B_{r}(0), 0)= deg_{N}(I-\widetilde{M}(0, \cdot), B_{r}(0), 0)=1$.
\end{center}
Then, from the existence property of degree, there exists  $u \in B_{r}(0)$ such that  $ u= \widetilde{M}(1,u)= M_{1}(u)=u$, which is a solution for (\ref{equa1}).  
\end{proof}

\begin{remark}
 In  \cite{chandra}, the nonlinear term $f(t, x,y)$ is bounded, in our result, the nonlinear term $f(t, x,y)$ may no more than a linear growth.
\end{remark}

\section{Boundary value problems in Hilbert spaces}
\label{work2}
Throughout this section, let  $(X,\left\langle \cdot ,\cdot \right\rangle )$ denote a real Hilbert space. Assume that $\varphi:X\rightarrow X$ satisfies the following conditions.  
\begin{enumerate}
\item  $\varphi^{-1}$  is a homeomorphism  with Lipschitz constant $k$.    
\item For any $x, y \in X, \  x\neq y $,
\begin{center}
$\left\langle \varphi(x)-\varphi(y), x-y\right\rangle >0$.
\end{center}
\end{enumerate} 

\begin{lemma}\label{tato2}
Let  $h\in C([0,T],\mathbb{R^{+}})$ be such that
\begin{center}
$\left\|f(t,x,y)\right\|\leq \left\langle f(t,x,y),x\right\rangle + h(t)$
\end{center}
for all  $(t,x,y)\in [0,T]\times X \times X $. If $(\lambda,u)\in [0, 1]\times C^{1}$ is  such that  $M(\lambda,u)=u$,  then there exists $ R> 0$ such that  $\left\| u \right\|_{1}\leq R$. 
\end{lemma}
\begin{proof}
Let $(\lambda,u)\in (0, 1]\times C^{1}$  be such that $M(\lambda,u)=u$. Using   \ref{igualdade}  we have that $u$ is solution  of (\ref{infi2}), which implies that
\begin{center} 
 $u'= \varphi^{-1}\left[\lambda H(N_f(u))\right]$, \   $u'(0)=0=u(T)$,
 \end{center} 
where for all  $t\in[0,T]$, we obtain
\begin{align*}
\left\|\lambda H(N_{f}(u))(t)\right\|& \leq \int_0^T \left\|f(s,u(s),u'(s))\right\|ds \\
&\leq \int_0^T \left\langle f(s,u(s),u'(s)), u(s)\right\rangle ds +\int_0^T h(s)ds. 
\end{align*}                          
On the other hand, because  $\varphi $  is a homeomorphism such that
\begin{center}  
$\left\langle \varphi(y), y\right\rangle\geq 0$
\end{center}
 for all $y\in X$. Then
\begin{center}   
$ \left\langle \varphi(u'(t))), u'(t)\right\rangle\geq 0  \ \   (t\in[0,T])$,     
\end{center}
and hence
\begin{center}
$-\int_0^T \left\langle \varphi(u'(t)), u'(t)\right\rangle dt \leq 0$.
\end{center}
Using the integration by parts formula and the boundary conditions, we deduce that
\begin{center}
$\int_0^T \left\langle (\varphi(u'(t)))', u(t)\right\rangle dt = - \int_0^T \left\langle \varphi(u'(t)), u'(t)\right\rangle dt \leq 0 $.
\end{center}
Since $\lambda\in (0, 1]$ and  $u$ is solution of (\ref{infi2}) we have that
\begin{center}
$\int_0^T \left\langle f(t,u(t),u'(t)), u(t)\right\rangle dt\leq 0$.
\end{center}
Hence,
\begin{center}
$\left\|\varphi(u'(t)) \right\|\leq \left\|h \right\|_{L^{1}}$.
\end{center}
It follows that there exists  $L>0$ such that  $\left\|u'\right\|_{\infty}\leq L$. Because $u\in C^{1}$ is such that  $u(T)=0$, we deduce that
\begin{center}
$\left\|u(t)\right\|\leq \int_t^T\left\|u'(s)\right\|ds \leq \int_0^T \left\|u'(s)\right\|ds\leq LT  \quad  (t\in \left[0,T\right])$,
\end{center}
and hence
\begin{center}
$\left\| u \right\|_{1}\leq R$, \  where  \  $R=$max$\left\{L, \ LT\right\}$.    
\end{center}
 Finally, if  $u=M(0, u)$, then  $u=0$, so the proof is complete.
\end{proof}

Now we show the existence of at least one solution for problem (\ref{equa1}) by means of Leray-Schauder degree.
\begin{theorem}
Let $f$  be completely continuous. Assume that $f$ satisfies the conditions of Lemma \ref{tato2}. Then  (\ref{equa1}) has at least one solution.
\end{theorem} 
\begin{proof}
Let $\lambda \in \left[0,1\right]$ and $u$ be a possible fixed  point of $M(\lambda, \cdot)$. Then, using Lemma \ref{tato2} we deduce that  
\begin{center}
$\left\| u \right\|_{1}\leq R$, \  where  \  $R=$max$\left\{L, \ LT\right\}$.    
\end{center}
Therefore, if $\rho>R$, it follows from the homotopy invariance of Leray-Schauder degree that $deg_{LS}(I-M(\lambda, \cdot), B_{\rho}(0), 0)$ is independent of $\lambda \in\left[0,1\right]$ so that, if we notice that $M(0, \cdot)=0$,
\begin{center}
$deg_{LS}(I-M(1,\cdot),B_{\rho}(0),0)=deg_{LS}(I -M(0,\cdot) ,B_{\rho}(0),0)=1$.
\end{center}
Hence there exists $u\in  B_{\rho}(0)$ such that is a solution for (\ref{equa1}). 
\end{proof}

Using a proof similar to that of Theorem \ref{numero2}, we obtain the following existence result.
\begin{theorem}\label{numero3}
 Let $f=f(t,x)$ be continuous. Assume that  $f$  satisfies the following conditions.   
\begin{enumerate}
\item  There exists $h\in C([0,1],\mathbb{R^{+}})$  such that
\begin{center}
$\left\|f(t,x)\right\|\leq \left\langle f(t,x),x\right\rangle + h(t)$
\end{center}
for all $(t,x)\in [0,1]\times X$.    
\item  $f$ sends bounded sets into bounded sets.	
\item For any a bounded set $S$ in $X$, 
\begin{center}
$ \alpha(f( \left[0, 1 \right] \times S))\leq k_{1}  \alpha(S)$, \  where  \  $0<k_{1} < 1/k$.
\end{center}
\end{enumerate} 
Then problem  (\ref{equa1}) has at least one solution.
\end{theorem}
\begin{proof}
Let $A$ be bounded in $C^{1}$. Applying Lemma \ref{lem5}  there exists $\tau \in \left[0, 1\right]$ or  $\omega \in \left[0, 1\right] $  with 
\begin{center}
$\alpha_{1}(M( \left[0, 1\right]\times A)) = \alpha(M(\left[0, 1\right]\times A)(\tau))$
\end{center}
or
\begin{center}
$\alpha_{1}(M(\left[0, 1\right] \times A)) = \alpha(\left(M(\left[0, 1\right]\times A)\right)'(\omega))$.
\end{center}
Proceeding as Theorem \ref{numero2}, we obtain in either case
\begin{center}
$\alpha_{1}(M(\left[0, 1\right] \times A))\leq  kk_{1}\alpha_{1} \left(A \right)$,  \  where \  $kk_{1}<1$.
\end{center}
Using the homotopy invariance of the degree for $\alpha$-condensing maps, we obtain 
\begin{center}
$deg_{N}(I-M(0,\cdot),B_{\rho}(0),0)=deg_{N}(I,B_{\rho}(0),0)=1$, \  where   $\rho> L$.
\end{center} 
Then, from the existence property of degree, there exists $u\in B_{\rho}(0)$ such that $u=M(1,u)$, which is a solution for (\ref{equa1}).
\end{proof}

The following corollary is concerned with the existence of one solution for (\ref{equa1}).
\begin{corollary}\label{numero4}
  Assume that  $f=f(t,x)$  satisfies the following conditions.   
\begin{enumerate}
\item Suppose that for any $\delta>0$  the mapping  $f:[0,1]\times X\rightarrow X$  is bounded and uniformly continuous
 in $ [0,1] \times  \overline{B_{\delta}(0)}$, where  $\overline{B_{\delta}(0)} =\left\{x\in X:\left\|x\right\|\leq \delta \right\}$.     
\item  There exists $h\in C([0,1],\mathbb{R^{+}})$  such that
\begin{center}
$\left\|f(t,x)\right\|\leq \left\langle f(t,x),x\right\rangle + h(t)$
\end{center}
for all $(t,x)\in [0,1]\times X$.  
\item There exists  a constant $k_{1}$  with  $0<k_{1}<1/k$ such that
\begin{center}
$\left\|f(t,x)-f(t,y)\right\|\leq k_{1}\left\|x-y\right\|$, \  for all $(t,x)\in [0,1]\times X$.
\end{center}
\end{enumerate} 
Then problem  (\ref{equa1}) has at least one solution.
\end{corollary}
\begin{proof}
Let $S$ be a bounded set in $X$. Let us consider 
\begin{center}
$H=\left\{\psi_{x}: x\in S\right\}$, \  where \  $\psi_{x}(t)=f(t,x)$ \  for all  \ $t\in \left[0,1\right]$.
\end{center}
Cleary, $H\subset C$, $H$ is bounded and equicontinuous. Thus, by  using the conclusion  of Lemma \ref{lem3}, we have 
\begin{center}
$\alpha_{c} (H)=\alpha (H([0,1]))=\alpha (f([0,1]\times S))=\displaystyle \max_{\left[0,1\right]} \alpha(\left\{f(t,x):x\in S\right\} )$.
\end{center}
Using the assumption 3, we obtain
\begin{center}
$\alpha (f([0,1]\times S))\leq   k_{1}\alpha(S)$.
\end{center}
By using the arguments of Theorem \ref{numero3}, we can obtain the conclusion of Corollary \ref{numero4}.
\end{proof}

\section*{Acknowledgements}
 This research was supported by CAPES and CNPq/Brazil.

\bibliographystyle{plain}

\renewcommand\bibname{References Bibliogr\'aficas}

\end{document}